\documentclass[11pt]{amsart}

\usepackage{amsmath,amssymb,amsfonts, mathtools,times,color,hyperref,fullpage}
\usepackage{csquotes}
\usepackage[section]{placeins} 

\usepackage{chngcntr}
\counterwithin{table}{section}

\title{sailing league problems}

\author{Robert Sch\"uler and Achill Sch\"urmann\\(University of Rostock)} 
\address{Institute of Mathematics, University of Rostock, 18051 Rostock, Germany}
\email{robert.schueler1989@gmail.com, achill.schuermann@uni-rostock.de}



\date{\today}
\newtheorem{defi}{Definition}[section]
\newtheorem{definition}[defi]{Definition}

\newtheorem{theorem}[defi]{Theorem}

\newtheorem{corollary}[defi]{Corollary}
\newtheorem{lemma}[defi]{Lemma}

\newcommand{\Z}{{\mathbb{Z}}}

\newcommand{\Nteams}{N_{\mbox{\normalfont\footnotesize teams}}}
\newcommand{\Nflights}{N_{\mbox{\normalfont\footnotesize flights}}}
\newcommand{\Ninflight}{N_{\mbox{\normalfont\footnotesize inflight}}}
\newcommand{\Nraces}{N_{\mbox{\normalfont\footnotesize races}}}
\newcommand{\Ninrace}{N_{\mbox{\normalfont\footnotesize inrace}}}

\newcounter{alg}

                       {\end{center}\end{figure}\medskip}

\begin{document}

\begin{abstract} 
  We describe a class of combinatorial design problems which typically occur in professional sailing league competitions.
  We discuss connections to resolvable block designs and equitable
  coverings and to scheduling problems in operations research. 
  We in particular give suitable boolean quadratic and integer linear optimization problem
  formulations, as well as further heuristics and restrictions, that can be used to solve sailing league problems in practice.
  We apply those techniques to three case studies obtained from real
  sailing leagues and compare the results with previously used tournament plans. 
\end{abstract}

\maketitle
 
\centerline{(Dedicated to Karen Aardal on occasion of her 60th birthday.)}

\section{Introducing sailing league problems}

There is a long history of human competitions finding the \emph{best} in a certain category. 
While in some disciplines (as long jump or marathons) the winner can be determined by direct measurement, other disciplines only provide the possibility to directly compare participants (as football or chess).
The latter are often organized in form of tournaments or leagues where the winner is found after a number of direct comparisons.
Several different forms of sport leagues and tournaments have been developed, each with its own pros and cons
(see for instance \cite{kendall_et_al_2010} and \cite{knust-2023}).

A sailing league typically consists of a fixed number of teams that meet on certain events,
often lasting a weekend or several days.
During each event, there are a number of rounds, called \emph{flights}, with a fixed number of equally sized races, 
such that every team competes exactly once during a flight. 
For instance, in the finals of the 2021 European Sailing Champions League
32 teams were scheduled to compete in 18 flights, each consisting of 4 races with 8 teams
(see \cite{www-ChampionsLeague}).
In the first and second national German sailing league
there are 6 events,
where 18 teams compete in 16 flights, each with 3 races of 6 teams 
(see \cite{www-SailingBundesliga}).
Note that there should be a certain number of teams in each race to make it more exciting,
but races with too many or even all teams are typically not desirable in a sailing league.
So the teams are typically partitioned into 2,3 or 4 equally sized races for a flight.

The resulting ranking (for which the rankings of each race are added) of any league should be chosen as fair as possible.
So for a better comparability of results, it is 
desirable that the number of races that a team is paired against another team does not vary too much.
In a \emph{perfect pairing list} each team would compete the same
number of times with any other team.
Whether or not such a perfect schedule exists is a special type of
\emph{block design problem}. 
We will discuss corresponding theoretical results and connections to
combinatorial design theory in Section~\ref{sec:DesignTheory}.
It turns out that such perfect pairing lists do rarely exist.
In particular, for the parameters of the sailing leagues mentioned above there is no such perfect schedule
(see Corollary~\ref{cor:no_perfect_pairing_list}).

If for given parameters it is not possible to find a perfect schedule, it is natural
to consider the following, more general {\em sailing league problem}: In it we ask
to minimize the difference between upper and lower bound for
the number of competitions any team 
competes with any other team.  
By this, the pairing list is as fair as possible in the sense that it
is as close as possible to an ideal pairing list
in which each team competes with every other team
the same number of times.
As we show in Sections~\ref{sec:programs} 
there are natural formulations of sailing league problems as boolean quadratic and as 
integer linear programming problems.

In Section~\ref{sec:CaseStudies} we apply the described techniques and 
take a closer look at three case
studies. We analyze in particular the Asian Pacific and
European Champions Leagues and the German National Leagues
in their 2021 season.
It turns out that schedules could have been improved with respect to
the utility function of the sailing league problem.
However, despite heavy computational efforts, some of the related 
combinatorial problems with larger parameters 
appear to be quite difficult and remain still open.

As in several other sports league problems there is a close connection
to problems in combinatorial design theory, as we explain in
Sections~\ref{sec:parameters} and~\ref{sec:DesignTheory}.
The sailing league problem studied here is also related to round-robin tournaments, 
where after a predefined number of races, in which each participant competes against every other participant, a winner is chosen. 
For a nice survey on round-robin tournaments we refer the interested reader to~\cite{rasmussen_trick2008-round-robin-survey}.
While in the classical round-robin tournament only two participants are paired with one another, 
a sailing race pairs multiple teams against each other.
While this is true for other disciplines too (as in games such as Skat for three players or Catan for four players), 
the number of teams in a sailing race is not predefined by the form of the competition 
and can be chosen freely by the organizers with respect to the available sailing track.
Consequently, the sailing league problem does not fit directly into the  
existing round robin literature, as it has recently been catagorized in \cite{vBGSG_2020}.
However, our methods to approach the problem, using known results from
combinatorial design theory and integer programming formulations, are 
similar to known techniques for solving different round robin tournaments
(see \cite{anderson_1997_book}, \cite{briskorn_2008_book} and the references within Section~\ref{sec:int-refs-sec}).

\section{Parameters and Examples of sailing league problems} \label{sec:parameters}

We start our survey with a precise introduction of \emph{sailing league problems}. Given three positive integer parameters $\Nteams, \Nflights$, $\Ninrace$, where $\Ninrace$ is a divisor of $\Nteams$, we call the following setup a \emph{pairing list}:

\begin{enumerate}
\item There is a set $T$ of competing teams of cardinality $\Nteams$.

\item There is a list $R$ of races. Each race is a subset of teams of cardinality $\Ninrace$. 

\item The races are organized in a list of flights $F$ of cardinality $\Nflights$ and
each team participates in each flight exactly once. 
I. e, $F$ is a partition of $R$ and each flight $f\in F$ is a partition of $T$.

\end{enumerate}

Note that we use the term \emph{list} to emphasize that it is allowed to repeat races and even flights.
Mathematically, one could also use an \emph{incidence structure} or a \emph{multiset} here.

We derive some further dependent parameters that turn out to be useful:

\begin{enumerate}

\item The {\em number of races in a flight}: $$\Ninflight = \frac{\Nteams}{\Ninrace}$$

\item
  The {\em total number of races}:
  \begin{equation} \label{eqn:Nraces}
       \Nraces = |R| = \frac{\Nteams\cdot\Nflights}{\Ninrace}
  \end{equation}
  
\item
  For each pair of teams $t, t'\in T$ we have their {\em number of pairings}
  $$\lambda(t, t') = |\{r\in R \ : \ t\in r \text{ and } t'\in r\}|.$$
  
\item
  Derived from the number of pairings, we have the \emph{minimal}, the \emph{maximal} and the \emph{average} number of pairings:
  \begin{align*}
      \lambda_{\min} &= \min_{\substack{t, t'\in T\\ t\neq t'}} \lambda(t, t')\\
      \lambda_{\max} &=\max_{\substack{t, t'\in T\\ t\neq t'}} \lambda(t, t')\\
      \lambda &= \frac{1}{\Nteams\cdot(\Nteams -1)}\cdot \sum_{\substack{t, t'\in T\\t\neq t'}} \lambda(t, t')
  \end{align*}

We easily obtain the following:
\begin{lemma}
\label{lem:lambda_formula}
We have
$$\lambda = \frac{\Nflights\cdot(\Ninrace -1)}{\Nteams -1}$$
\end{lemma}
\begin{proof}
    Let $t\in T$ be a fixed team. 
    In each flight, $t$ participates in exactly one race against $(\Ninrace -1)$ opposing teams.
    Therefore, there is a total of $\Nflights\cdot(\Ninrace -1)$ pairings for each team.
    Since there are $(\Nteams -1)$ opposing teams we get the claimed identity for the average number of pairings.
\end{proof}
  
\end{enumerate}

\bigskip

For parameters $\Nteams, \Nflights$, $\Ninrace$ the corresponding \emph{sailing league problem} asks to find a pairing list with $\lambda_{\max} - \lambda_{\min}$ as small as possible.

\bigskip

Note that $\lambda_{\min}$ and $\lambda_{\max}$ are integers by definition, while $\lambda$ may not be an integer.
If we manage to find a pairing list with $\lambda = \lambda_{\min} = \lambda_{\max}$ and therefore $\lambda_{\max} - \lambda_{\min} = 0$, we call this pairing list \emph{perfect}. 
Obviously, a perfect pairing list is also optimal for the respective sailing league problem.

\bigskip

In order to have a compact human readable way to present different pairing lists of sailing leagues, we introduce the concept of a \emph{tournament plan} for a sailing league:

\begin{enumerate}
    \item A tournament plan is a matrix with $\Nflights$ rows and $\Nteams$ columns, where each row corresponds to a flight and each column corresponds to a competing team.
    \item Each flight consists of $\Ninflight$ races. We identify each race in a flight by a number between $1$ and $\Ninflight$.
    \item The entry in the tournament plan corresponding to team $t$ and flight $f$ is the number of the race (in $f$) that team $t$ competes in. Since $f$ is a partition of the teams, this number is uniquely determined.
\end{enumerate}

\bigskip

To make the theoretical discussions of this paper more vivid, we take a look at 
three different examples of real world sailing leagues.
Exemplarily, their parameters are chosen in such a way, that different techniques for their 
solution can be applied, respectively different phenomena can be explained.

In case of the national German (first and second) Sailing Leagues there
are six events 
with the scheduling parameters
\begin{equation} \label{eqn:parametersGNL}
\Nteams = 18,  
\quad
\Nflights = 16,
\quad
\Ninrace = 6
\quad
\mbox{and}
\quad
\Ninflight = 3,
\quad
\Nraces = 48,
\quad
\end{equation} 
for each event.
In contrast, events of the European Sailing Champions Leagues are realised with various different parameters.
Under several possibilities, we study the final of the European Sailing Champions League 2021 in Porto Cervo with parameters
\begin{equation} \label{eqn:parametersESCL21}
\Nteams = 32,  
\quad
\Nflights = 18,
\quad
\Ninrace = 8
\quad
\mbox{and}
\quad
\Ninflight = 4,
\quad
\Nraces = 72
.
\end{equation} 
Despite the $32$ possible slots for teams, in the actual event, only $29$ teams competed.
This gives a motivation to look at types of
relaxed sailing league problems, which we discuss briefly in
Section~\ref{sec:relaxations}.

For the sake of small parameters we will also look at a pairing list
for the final event in Newcastle, Australia, of the Asian Pacific Sailing Champions League in 2021.
The parameters were
\begin{equation} \label{eqn:parametersSCLAP21}
\Nteams = 10,  
\quad
\Nflights = 8,
\quad
\Ninrace = 5
\quad
\mbox{and}
\quad
\Ninflight = 2,
\quad
\Nraces = 16
.
\end{equation}
They used a base pairing list in which each team is in only 8 races/flights;
see Table~\ref{tab:tournament-plan-SCLAP21}.

\begin{table}[h]
	\begin{tabular}{c|cccccccccc}
		Flight/Team& 1& 2& 3& 4& 5& 6& 7& 8& 9& 10 \\
		\hline
		1 &1&1&1&1&1 &2&2&2&2&2\\
		2 &1&1&1&2&2 &2&1&2&1&2\\
		3 &1&2&2&1&1 &2&1&2&1&2\\
		4 &1&1&2&2&1 &2&1&2&1&2\\
		5 &1&1&1&2&1 &2&2&2&1&2\\
		6 &2&1&1&2&1 &2&1&2&1&2\\
		7 &2&1&1&1&2 &2&1&2&1&2\\
		8 &2&1&1&1&1 &2&1&2&2&2\\
	\end{tabular}
	\bigskip
	
	\caption{The Asia-Pacific Champions-League tournament plan
  2021 in Newcastle, Australia.
  }
	\label{tab:tournament-plan-SCLAP21}
	
\end{table}
	

In the actual event this base schedule was to be repeated up to 6 times.
So all parameters of the sailing league problem stay the same,
except for $\Nraces$ and $\Nflights$, which are multiplied by~6.

In Section~\ref{sec:CaseStudies} we will give an optimal solution 
for the sailing league problem of the German national leagues as well as for the
Asia-Pacific Champions-League with $8$ and $16$ flights. 
The parameters of the European Champions-League 
seem to make a solution very hard and the sailing league problem remains open in this case.

\section{Perfect pairing lists and connections to Design Theory} \label{sec:DesignTheory}

A simple necessary condition for the existence of a perfect pairing list is easily found:

\begin{lemma}
    If we have a perfect pairing list with parameters $\Nteams$, $\Nflights$, $\Ninrace$, then
    $\lambda$ is an integer.
    In particular, $(\Nteams -1)$ is necessarily a divisor of $\Nflights\cdot(\Ninrace -1)$.
\end{lemma}
\begin{proof}
    In a perfect pairing list, we have $\lambda_{\min} = \lambda = \lambda_{\max}$ and thus $\lambda$ clearly is an integer.
    The rest follows by Lemma \ref{lem:lambda_formula}.
\end{proof}

By this result, we immediately have negative news for the parameters of our three study cases:

\begin{corollary}
\label{cor:no_perfect_pairing_list}
  There is no perfect pairing list with the parameters of the national German Sailing Leagues, the European Sailing Champions League and the Asian Pacific Sailing Champions League respectively.
\end{corollary}
\begin{proof}
  The corresponding values of $\lambda$ are $\frac{80}{17}$, $\frac{126}{31}$ and $\frac{32}{9}$ respectively, which are not integers.
\end{proof}

\bigskip

We note that it is not difficult to find \emph{any} perfect pairing list -- 
even when some parameters are given.
We can for example do one of the following:

\begin{enumerate}
    \item[(a)]
           Let all teams participate in every race, i.e. $\Ninrace = \Nteams$ and $\Nraces = \Nflights$.
    \item[(b)] 
           For given $\Nteams$ and $\Ninrace$ perform \emph{all possible} flights.
\end{enumerate}

However, these trivial perfect pairing lists are often not useful in practice.
A race between too many participants as in (a)
is apparently hard to realize under fair conditions. 
That is the main reason why sailing leagues are organized as they are.
The other possibility in (b), 
leads to an extremely high number of flights, due to a combinatorial explosion.
We have a total of 
$$\frac{\Nteams !}{(\Ninrace !)^{\Ninflight} \cdot \Ninflight!}$$
flights, where we use $\Ninflight=\Nteams/\Ninrace$.
For example, if we have $\Nteams=18$ teams and $\Ninrace=6$ teams per race, as in
the national German Sailing Leagues, the pairing list with all possible flights
would have a total of $2,858,856$ flights. 

\bigskip

Problems similar to finding a perfect paring list have been studied in
mathematical \emph{design theory}. This theory is also used in other
contexts to create team schedules or organizing experiments. For a
nice introduction to design theory, we refer to
\cite{beth_jungnickel_lenz1999design_theory}  and \cite{beth_jungnickel_lenz1999design_theoryII}.
A very good overview is provided by \cite{HandbookCombDesigns_2006}.

One important concept of design theory are \emph{block designs}
(sometimes also called \emph{balanced incomplete block designs} and abbreviated as \emph{BIBD}).
For a block design, assume that we have a finite set whose elements are called \emph{points},
  together with a collection of \emph{blocks} which are subsets of the points.
Each block contains the same number of points as every other block,
and each pair of points is also contained in the same number of blocks.

If we interpret points to be the teams participating in the
sailing league and the blocks to be the different races, it is
apparent that finding a block design and a perfect pairing list are two strongly related problems.

For our further discussion, 
we first give a precise definition of block designs, 
using the naming conventions of this paper:

\begin{definition}

	Let $\Nteams,\Ninrace,\lambda$ be positive integers.
	A \emph{$(\Nteams,\Ninrace,\lambda)$-block design} is a tupel $(T, R)$
	where $T$ is a set of teams and $R$ is a list (an incidence structure) of subsets of $T$, such that:

	\begin{enumerate}
        \item $|T| = \Nteams$,
        \item $|r| = \Ninrace$ for all $r\in R$,
        \item $\lambda = \lambda(t, t')$ for all $t, t' \in T$, $t\neq t'$.
	\end{enumerate}
\end{definition}

The elements of $T$ are also called \emph{points} of the design,
while the elements of $R$ are called \emph{blocks} of the design.
Note that in design theory literature, typically the cardinality of the pointset is denoted by $v$, the block size is $k$ and therefore the 
notation of a $(v, k, \lambda)$-design is typically used.
With our notation here we emphasize the connection to pairing lists.

Note that a pairing list is neither a stronger nor a weaker concept than a block design, since on the one hand it allows for different values of $\lambda(t, t')$ while on the other hand the races have to be organized in flights. 
From a perfect pairing list we also get a block design.
However, from an arbitrary block design 
we do not necessarily obtain a perfect pairing list,
as it is not always possible
to group the races into equally sized flights
-- even if the block (race) size $k=\Ninrace$ divides the number of teams $v=\Nteams$.

There are design theoretic concepts to attack this problem.
A subset of $R$ is called a \emph{parallel class} or \emph{resolution class} of the design,
if it is a partition of $T$. If the races can be partitioned into
parallel classes, the block design is called \emph{resolvable}
(sometimes abbreviated to RBIBD, standing for \emph{resolvable
  balanced incomplete block designs}, see \cite{Handbook_ChapterResolvableDesigns_2006}).
As a consequence, a perfect pairing list
with parameters $\Nteams$, $\Nflights$, $\Ninrace$ (and $\lambda_{\min} = \lambda = \lambda_{\max}$)
describes the same concept as
a resolvable ($\Nteams$, $\Ninrace$, $\lambda$)-block design,
and the terms \emph{flight} and \emph{parallel class} are interchangeable.

If we require from a pairing list that every pair of teams competes at
least once in some flight ($\lambda_{\min}\geq 1$), we are asking for a block design that is a 
\emph{resolvable covering of pairs}
(see \cite{Handbook_ChapterCoverings_2006}, \cite{Handbook_SchedulingATournament_2006}). 
Asking for a minimum number of parallel classes (flights)
to obtain such a resolvable covering of pairs has been studied in \cite{Haemers-1999}.
With the additional requirement that every pair of teams competes
exactly once or twice against each other a resolvable covering of pairs
is called \emph{equitable} (see \cite{DHP-2003}).
In this case, we have $1\leq \lambda_{\min} \leq \lambda_{\max}\leq 2$
and therefore $\lambda_{\max}-\lambda_{\min}\leq 1$.
Hence, the sailing league problem is a generalization of an equitable
covering in a sense.  
In design theoretic terms we are asking for a resolvable covering of pairs minimizing $\lambda_{\max}-\lambda_{\min}$
(is \emph{``as equitable as possible''}), for a given 
number of points (teams) $v=\Nteams$, block (race) size $k=\Ninrace$ and parallel classes (flights).

\bigskip

What can we use from the theory of resolvable block designs to 
solve our sailing league problems in general?
On the one hand, we can certainly 
obtain existence or nonexistence results of perfect pairing lists 
for certain parameters from corresponding results about resolvable block designs.
On the other hand, we can use the existence of such designs 
for finding fairly good pairing lists with slightly 
different parameters.

We start with some necessary conditions on parameters for the existence 
of block designs. We note that the formula of Lemma~\ref{lem:lambda_formula} remains
valid for general block designs, even if the races (blocks) 
can not be grouped into flights (parallel classes).
The parameter 
$$\Nflights=\frac{\lambda\cdot (\Nteams-1)}{\Ninrace-1}$$
is in general equal to the number of races that contain a given team
-- and each team in a block design is contained in the same number of races
(\cite[Thm. 2.10.]{beth_jungnickel_lenz1999design_theory}).
Using this equality in \eqref{eqn:Nraces} we obtain another known formula
for a $(\Nteams,\Ninrace,\lambda)$-block design:  
$$\Nraces = \frac{\lambda\cdot \Nteams\cdot(\Nteams-1)}{\Ninrace\cdot(\Ninrace-1)}$$ 
holds for the number of races (blocks) of the design.
Since the quantities $\Nraces$ and $\Nflights$ have to be integers, 
we can easily deduce the following:

\begin{corollary}{({\cite[Cor. 2.11.]{beth_jungnickel_lenz1999design_theory}})}

	If there exists an $(\Nteams, \Ninrace,\lambda)$-block design, then

	\begin{enumerate}
		\item 
		  $(\Ninrace-1)\;$ divides $\;\lambda\cdot (\Nteams-1)$.
		\item
		  $\Ninrace\cdot(\Ninrace-1)\;$ divides $\;\lambda\cdot\Nteams\cdot(\Nteams-1)$.
	\end{enumerate}

\end{corollary}

Note that in the case that $\Ninrace$ divides $\Nteams$, only the first condition is relevant as the second follows directly.

In case we have a perfect pairing list, the number of races (blocks) each team competes in is equal to the number of flights (parallel classes). 
We can find another necessary conditions for perfect pairing lists from the theory of resolvable block designs:

\begin{corollary}
  {(\cite[Equation 1.20]{bose1942note})}
\label{cor:necessary_conditions_resolvable_block_design}

	If there exists a perfect pairing list with parameters $\Nteams$, $\Nflights$, $\Ninrace$, then

		$$\Nraces \geq \Nteams + \Nflights -1.$$

\end{corollary}

\bigskip

Can we hope for a perfect pairing list if we slightly
modify the parameters in our three examples?
For this, we examined all possible nearby parameters, in ranges 
$\Nteams \pm 3$, $\Nflights \pm 3$, $\Ninrace \pm 2$, where $\Nteams$, $\Nflights$, $\Ninrace$ are set to be the parameters of the German Sailing Leagues, the European Champions Leagues and the Asian Pacific Sailing Champions Leagues respectively.
After applying the necessary conditions above, there are only $13$ cases left to discuss. Most of these cases can then be handled by results from design theory literature. A summarized discussion of these parameters is displayed in 
Table~\ref{tab:legal_params}. For the convenience of the reader who hopes to apply these findings directly, 
we present related tournament plans of known existing 
perfect pairing lists in Appendix\ref{app:perfect_pairing_lists}
(Tables~\ref{tab:tournament-plan-8teams-7flights-4inrace} to~\ref{tab:tournament-plan-20teams-19flights-5inrace}).

\begin{table}

	\begin{tabular}{cccc|ccc}
		$\Nteams$&$\Nflights$&$\Ninrace$&$\lambda$&ppl?&Table&Reference\\
		\hline
		8&7&4&3&Yes&\ref{tab:tournament-plan-8teams-7flights-4inrace}&\cite[Table 9.1]{beth_jungnickel_lenz1999design_theoryII}\\
        9&8&3&2&Yes&\ref{tab:tournament-plan-9teams-8flights-3inrace}& \cite[Table 9.1]{beth_jungnickel_lenz1999design_theoryII}\\
        10&9&5&4&No&&\cite[Lemma 2.10]{miao1994existence-of-resolvable-bibds}\\
        12&11&3&2&Yes&\ref{tab:tournament-plan-12teams-11flights-3inrace}&\cite[Theorem 4]{hanani1974-on-resolvable-balanced-incomplete-block-designs} \\
        12&11&4&3&Yes&\ref{tab:tournament-plan-12teams-11flights-4inrace}&\cite[Table 1]{baker1983resolvable_bibd_and_sols}\\
        12&11&6&5&Yes&\ref{tab:tournament-plan-12teams-11flights-6inrace}&\cite[Table 3]{baker1983resolvable_bibd_and_sols}\\
		\hline
        15&14&5&4&No&&\cite{kasi-ostergard2001-there-exists-no-15-5-4-RBIBD}\\
        16&15&4&3& Yes&\ref{tab:tournament-plan-16teams-5flights-4inrace}&{\cite[Thm. 6.9.2.]{calinski_tadeusz_kageyama_block_designsII}}\\
        16&15&8&7&Yes&\ref{tab:tournament-plan-16teams-15flights-8inrace}& {\cite[Thm. 6.9.2.]{calinski_tadeusz_kageyama_block_designsII}}\\
        18&17&6&5&Yes&\ref{tab:tournament-plan-18teams-17flights-6inrace}&{\cite{Kageyama_1983},\cite[Table
                                                                           3]{baker1983resolvable_bibd_and_sols}}\\
        20&19&4&3&Yes&\ref{tab:tournament-plan-20teams-19flights-4inrace}&\cite[Theorem 6.5]{baker1983resolvable_bibd_and_sols}\\
        20&19&5&4&Yes&\ref{tab:tournament-plan-20teams-19flights-5inrace}&\cite[Lemma 2.10]{miao1994existence-of-resolvable-bibds}\\
        \hline
        35&17&7&3&unknown\\
	\end{tabular}
	
	\bigskip
	
	\caption{Parameters near the parameters of the three exemplary sailings leagues that satisfy all necessary conditions. The column labeled \emph{ppl?} indicates whether there exists a perfect pairing list with these specific parameters.
    }
	\label{tab:legal_params}
	
\end{table}

It turns out that for the parameters of the German national league and the Asian Pacific Champions League, there are interesting \enquote{nearby results} in 
Table~\ref{tab:legal_params}.

For the German national leagues we have an existence result 
for a corresponding perfect pairing list, with just one flight more and otherwise the same parameters. By removing an arbitrary flight from this perfect pairing list 
of Table~\ref{tab:tournament-plan-18teams-17flights-6inrace},
we easily obtain an optimal pairing list with $\lambda_{\max} - \lambda_{\min} = 1$.
The proof of optimality follows from Corollary \ref{cor:no_perfect_pairing_list}.

For the Asian Pacific Champions League we have a nearby nonexistence
result of a $(10, 9, 4)$-block design,
for a putative perfect paring list (\cite[Lemma 2.10]{miao1994existence-of-resolvable-bibds}), 
with again just one flight more and otherwise the same parameters.
This proves that there can not exist a pairing list with the parameters
of the Asian Pacific Champions League having 
$\lambda_{\max} - \lambda_{\min} = 1$.
Otherwise we could add an arbitrary flight to an existing perfect pairing list 
and obtain this utility value.
As we will explain in the next section, 
it is possible to show $\lambda_{\max} - \lambda_{\min} = 3$
using suitable optimization software.
Note that this in return gives an alternative computational proof of the 
nonexistence result in~\cite{miao1994existence-of-resolvable-bibds}.

We end this section by subsuming the above observation on removing or adding arbitrary flights to an (putative) existing pairing list:

\begin{lemma}
\label{lem:add_or_remove_rows}
    Let $P$ be a pairing list for parameters $\Nteams$, $\Nflights$, $\Ninrace$ and let $k$ be an integer.
    If $k < \Nflights$, then there exists a pairing list $P'$ for parameters
    $\Nteams$, $\Nflights -k$, $\Ninrace$
    with
    $$\lambda_{\max}(P) - \lambda_{\min}(P) -k \; \leq \; \lambda_{\max}(P') - \lambda_{\min}(P') \; \leq \; \lambda_{\max}(P) - \lambda_{\min}(P) + k.$$
    Furthermore, for arbitrary $k$, there also exists a pairing list $P''$ for parameters $\Nteams$, $\Nflights + k$, $\Ninrace$ with 
    $$\lambda_{\max}(P) - \lambda_{\min}(P) -k \;  \leq \; \lambda_{\max}(P'') - \lambda_{\min}(P'') \; \leq \; \lambda_{\max}(P) - \lambda_{\min}(P) + k.$$
    
    If in addition $\Nteams > \Ninrace$, $k=1$ and $P$ is a perfect pairing list, 
    then $P'$ as well as $P''$ are optimal with 
    $$\lambda_{\max}(P') - \lambda_{\min}(P') \; = \; 1 \; = \; \lambda_{\max}(P'') - \lambda_{\min}(P'').$$
\end{lemma}

\begin{proof}
    Consider removing $k$ arbitrary flights from $P$ to obtain $P'$ and adding $k$ arbitrary flights to $P$ to obtain $P''$. Thereby, we alter the numbers of times two teams have been paired against each other by at most $k$, yielding the inequalities.

    If $P$ is a perfect pairing list and $k=1$, then the inequalities yield
    $\lambda_{\max}(P') - \lambda_{\min}(P'), \lambda_{\max}(P'') - \lambda_{\min}(P'')\leq 1$ and for the assertion to be true we are left to exclude that $P'$ and $P''$ are perfect pairing lists.
    Suppose $P'$ is perfect.
    By Lemma \ref{lem:lambda_formula} it follows that, $\lambda(P) = \frac{\Nflights\cdot (\Ninrace -1)}{\Nteams -1}$ and $\lambda(P') = \frac{(\Nflights -1)\cdot (\Ninrace -1)}{\Nteams -1}$ are integers.
    By
    $$\lambda(P) - \lambda(P') = \frac{\Ninrace-1}{\Nteams -1} \in \Z$$
    this yields that $\Nteams -1$ has to be a divisor of $\Ninrace -1$. This is impossible since $\Nteams > \Ninrace$ yielding a contradiction. A similar argument can be found for $P''$.
\end{proof}

\section{Quadratic and linear programs for the sailing league problem}
\label{sec:programs}

As in our three study cases, there does not exist a perfect pairing list 
for most combinations of possible parameters.
In this section we therefore describe approaches to formulate 
any sailing league problem with given parameters 
as an optimization problem. They can -- at least in principle -- 
be solved with a suitable software tool for operations research. 
We used Gurobi \cite{www-gurobi} and Cplex \cite{www-cplex} to compare the different programs.
We first give a natural boolean quadratic problem formulation and then describe
two different linear programs that appear to be better solvable in practice.
Depending on the given parameters one or the other may perform better,
as it is also shown by our case study in Section~\ref{sec:CaseStudies}.

We identify in the following each 
team, race and flight by its index: We use
$k,l\in\{1,\dots, \Nteams\}$ for team indices and the index $j\in\{1, \dots, \Nraces\}$ 
for the race. Index $i\in\{1, \dots, \Nflights\}$ is used for the flight. 
Furthermore, we assume that the races are organized in such a way, that $\Ninflight$ consecutive races form a flight, i.e., flight $i$ consists of the races with indices $\{(i-1)\cdot \Ninflight +1, \dots, i\cdot \Ninflight\}$. 

Two conditions of a pairing list are already satisfied by the above conventions:
\begin{enumerate}
    \item There are $\Nteams$ teams.
    \item The flights form a partition of the races.
\end{enumerate}

Other conditions are to be modelled in the respective programs below.

\bigskip

\subsection{A boolean quadratic program} \label{sec:BooleanQuadratic} 

Allowing quadratic constraints there is a simple and straightforward
formulation of the sailing league problem.
For it we can use boolean, respectively $0$-$1$ decision variables
$b_{jk}\in \{0,1\}$ indicating if team $k$ participates in race $j$.

We model the condition, that the number of teams in each race has to be  
equal to the constant $\Ninrace$ by 
the following 
linear equations for each race $j$.
Since there is a fixed number of teams in each race, we get the constraints
$$
  \sum_{k=1}^{\Nteams} b_{jk} = \Ninrace
$$
for every race $j$. 

Additionally, we model the condition that every flight forms a partition of the teams (each team is in exactly one race per flight) 
by the following linear equation for every team $k$ and every flight $i$:
$$
  \sum_{j=(i-1)\cdot \Ninflight + 1}^{i\cdot \Ninflight }  
  b_{jk} = 1
$$

\bigskip

In order to set up a linear utility function, 
we use two variables $a$ and $b$ as lower and upper bound for the number of
competitions between two different teams $k$ and $l$ during a sailing league event.
We obtain the two nonlinear quadratic inequalities
\begin{equation}  \label{eqn:banda_bounds}
a \leq
\sum_{j=1}^{\Nraces} b_{jk} \cdot b_{jl}
\leq b
\end{equation}
for every pair of teams $k,l$, where we can assume $l<k$ due to symmetry.
Note that the quadratic $0$-$1$-expressions $b_{jk} \cdot b_{jl}$
indicate if team~$k$ and team~$l$ compete in race~$j$.
Thus we may assume that $a$ and $b$ are integers.

Our goal is now to minimize the linear utility function $b-a$.
Note that for a pairing list that optimizes this program, we have $\lambda_{\max} = b$, $\lambda_{\min} = a$ and $\lambda_{\max} - \lambda_{\min} = b -a$ and thus the program is equivalent to the sailing league problem.

\bigskip

The described quadratically constrained program 
can be implemented in a straightforward way.
However, 
further simplification, respectively a reduction of used variables is possible,
by using some additional symmetry breaking assumptions:
We can require for instance that the teams in race~$1$ 
are the teams $1$, $2$, \ldots $\Ninrace$ and so on...
Thus we may assume
$$
b_{1k} = 1 \; \mbox{for $k\leq \Ninrace$ and $b_{1k} = 0$ otherwise,}
$$
etc..
Additionally we can also require that team~$1$ is in the first race of
every flight.
So $b_{jk} = 1$
for the first race of each flight, i.e. for all $j=(i-1)\cdot \Ninflight + 1$.

\subsection{An integer linear program}   \label{sec:IntegerLinear}  

For practical computations it is usually desirable to get rid of the 
quadratic constraints.
With the $0$-$1$ decision variables in the 
quadratic problem formulation of Section \ref{sec:BooleanQuadratic} 
we can easily set up a purely linear problem formulation.
This has the price of additional variables though:
We may use variables $b_{jkl}$ for the non-linear expressions $b_{jk} \cdot b_{jl}$
in the above quadratic problem formulation.
The $0$-$1$ variables $b_{jkl}$ are thus indicating 
if team~$k$ and team~$l$ compete in race~$j$.

The relation 
$
b_{jkl} = b_{jk} \cdot b_{jl}
$
among the variables can be put into linear constraints
for every race $j$ and every pair of teams $k,l$ with $l<k$:
We use the three linear inequalities
\begin{align}
\label{eq:three_ineq}
b_{jk} + b_{jl} - 1 \; \leq \; b_{jkl} \; \leq \; b_{jk} \; , \; b_{jl}
.
\end{align}
This is the \emph{standard linearization} by Glover and Woolsey \cite{glover1974converting}. 
In it we have a total of
$
3 \cdot \Nraces \cdot {\Nteams \choose 2}
$
additional linear inequalities and
$\Nraces \cdot {\Nteams \choose 2}$ additional variables
$b_{jkl}$. 
We note that these additional variables can be chosen either binary 
or as real, but then with additional bounds $0\leq b_{jkl} \leq 1$.
Then the variables are forced to be $0$ or $1$ by the inequalities \eqref{eq:three_ineq}.
Using real variables appears to be beneficial when used
with the standard parameters of \cite{www-gurobi} or \cite{www-cplex}.

As in the quadratic problem formulation we want to minimize $b-a$
under the now linear side constraints (see \eqref{eqn:banda_bounds})
$$
a \leq
\sum_{j=1}^{\Nraces} b_{jkl}
\leq b,
$$
for every pair of distinct teams $k,l$, with $l<k$.

As a rule of thumb, this linear problem formulation works best 
in practice with current solvers 
for suitable small parameters of the sailing league problem
-- despite its large number of variables.

\subsection{A smaller integer linear program}
\label{sec:small_int_prog}

Since every team competes in every flight only once, we can 
reduce the number of variables in the linear program formulation above.
Instead of using variables $b_{jkl}$ to indicate whether team $k$ competes in race~$j$ against team $l$, we can use $0$-$1$ variables $b_{ikl}$ indicating if
team $k$ competes in flight~$i$ against team $l$.
We require for each flight $i$ and each team $k$ that 
$$
\sum_{l=1, l\not=k}^{\Nteams} b_{ikl} \; = \; \Ninrace - 1
,
$$
that is, the number of teams competing against team $k$ in flight $i$ 
is fixed.
Additionally we require that $b_{ikl}=1$ if $b_{ikm} = 1$ and $b_{ilm}=1$, 
for every flight $i$ and all triples of distinct team indices $k,l,m$.
This can be achieved by linear inequalities
$$
b_{ikm} + b_{ilm} - 1 \; \leq \; b_{ikl}
$$
for every flight $i$ and all possible triples $k,l,m$.
Note that we may use variables $b_{ikl}$ with $l<k$ only in all of these inequalities, 
because of the symmetry $b_{ikl}=b_{ilk}$.

From variables satisfying all the described inequalities we can easily construct a tournament plan. In it, the order of races in each flight can be chosen arbitrarily.
In this way, this second linear problem formulation gives us a solution of a race plan up to symmetry of race numbering in each flight.
As in the other two problem formulations we want to minimize $b-a$,
this time under the linear side constraints (see \eqref{eqn:banda_bounds})
$$
a \leq
\sum_{i=1}^{\Nflights} b_{ikl}
\leq b,
$$
for every pair of distinct teams $k,l$, with $l<k$.

Overall, this second linear problem formulation has the advantage of less variables 
compared with the one before. On the other hand it usually has more side constraints,
as their number grows cubic with the number of teams -- compared to quadratic growth in the other linear problem formulation.
Concretely, in the second linear program we have a number of side constraints in the order 
$O(\Nflights \cdot \Nteams^3)$ while in the first linear program the number is in the order of $O(\Nraces \cdot \Nteams^2)$.
An additional practical disadvantage of this second linear problem formulation is the fact that we can not simply relax the variables $b_{ikl}$ to reals between $0$ and $1$ -- as it is possible for the variables $b_{jkl}$ in the first linear program. Here, in contrast to the first linear program, we do not have inequalities as in \eqref{eq:three_ineq} which force the variables to be $0$ or $1$.

However, together with the advantage of less variables, in our experiments we also discovered a second useful advantage: If we relax some of the variables $b_{ikl}$ to reals between $0$ and $1$, we obtain a possibly smaller minimum $b-a$ that can be a useful lower bound! Since we know that $b-a$ is an integer, such a bound can even be used to prove a tight bound on $b-a$, even if the corresponding $0$-$1$-problem can not be solved.
We use this \emph{relaxation trick} in the next section to give a solution of the 
Asian Pacific Champions League with $16$ flights.

\subsection{Integer linear programs for other sport leagues}  \label{sec:int-refs-sec}

Integer linear programming is probably one of the most used 
techniques in sports league scheduling.
To the best of our knowledge, the sailing league problem has not been
studied in the corresponding OR literature.
Nevertheless, there is a long list of publications using 
in particular integer linear programs to schedule
different sport leagues.
Naturally, in many of these some similarities to our two problem formulations can be found.
Without being near to a complete list, we like to mention the
following publications in this context (in chronological order):
\cite{Robinson1991},
\cite{FleurentFerland1993},
\cite{NemhauserTrick1998},
\cite{Trick2004},
\cite{DellaCroceOliveri2006},
\cite{Rasmussen_2008},
\cite{RasmussenTrick2009},
\cite{KostukWilloughby2012},
\cite{Recalde_etal_2013},
\cite{Alarcon2014},
\cite{GoerigkWestphal2016},
\cite{Duran_etal_2017},
\cite{Cocchi_etal2018},
\cite{Bouzarth_etal2021}
\cite{vBulck_Goossens_2023},
\cite{vDHLS_2023}  
.
We also like to mention that differennt contraint programming and SAT techniques
may work well in some cases, as it is suggested in papers like
\cite{Smith_etal_1996} or
\cite{Baker_etal_2002}.
We ourselves did not test corresponding approaches so far, but we
think this could be a fruitful direction for future research on sailing league problems.

\section{Three Case Studies}   \label{sec:CaseStudies}

Using the tools of Sections~\ref{sec:DesignTheory} and~\ref{sec:programs},
we take a closer look at the three sailing league events mentioned 
in Section~\ref{sec:parameters}.

\subsection{Asian Pacific Sailing Champions League}

We start with the smallest event, involving only $10$ teams.
In 2021 the final event in Newcastle, Australia, 
was realized with the pairing list given in Table~\ref{tab:tournament-plan-SCLAP21}.
A quick analysis of this base pairing list, with $8$~flights and $16$~races, 
reveals that already this schedule is particularly bad for the sailing league
problem.
We note that other criteria were seemingly more important for the schedule,
but it is not clear to us which precicely and why.
Teams~6, 8 and 10, for instance, compete against each other in all 8
flights, while Team~2 competes against these three teams only once. I.e. $\lambda_{\min} = 1$, $\lambda_{\max} = 8$, $\lambda_{\max} - \lambda_{\min} = 7$.
Multiplying the base schedule by 6 --- as it was planned in the actual event --- 
does even amount in 48 versus 6 matches for the same teams.

With the given parameters we can actually solve all linear/quadratic problems presented in Section~\ref{sec:programs} on a standard laptop. Fastest solutions with \cite{www-gurobi}
and \cite{www-cplex} are obtained when we use the first linear program of 
Section~\ref{sec:IntegerLinear}, with as many variables as possible 
relaxed to real variables between $0$ and $1$. Note that not only all the variables 
$b_{jkl}$ can be relaxed in such a way without changing the outcome, 
but also some of the variables $b_{jk}$ that are shared with the quadratic program.
I.e., for every equation we can relax one of the involved variables to be real. 
Overall we obtain a computational proof of the following result:

\begin{theorem}
\label{thm:asian_pacific}
For the parameters $\Nteams = 10$, $\Nflights = 8$ and $\Ninrace = 5$
an optimal tournament plan has a utility value $\lambda_{\max} - \lambda_{\min} = 3$
and satisfies $\lambda_{\min}=2, \lambda_{\max}=5$. 
A corresponding plan is contained in the first eight flights (rows) of
Table~\ref{tab:tournament-plan-10teams-16flights-5inrace}.
\end{theorem}

Note that for the proof it is not enough to compute 
$\lambda_{\max}-\lambda_{\min}=3$ with the given parameters:
Since the average of pairwise competitions in the 
Asian Pacific Champions League is $\lambda=32/9=3.555\ldots$,
it could in principle also be possible that there exist tournament plans with
lower and upper bounds $\lambda_{\min}=1, \lambda_{\max}=4$ or $\lambda_{\min}=3, \lambda_{\max}=6$.
However, corresponding integer linear programs with these values of $a$ and $b$ set fixed turn out to be infeasible, that is, there do not exist corresponding paring lists.

Note that Theorem \ref{thm:asian_pacific} together with Lemma
\ref{lem:add_or_remove_rows}
also implies the nonexistence of a $(10, 9, 4)$-block design as has
been proven previously using combinatorics
(see \cite[Lemma 2.10]{miao1994existence-of-resolvable-bibds}).

\bigskip

The computed optimal schedules with $\lambda_{\min}=2$ and $\lambda_{\max}=5$ are still quite uneven, respectively the minimum of the utility function appears
quite big. This is in particular the case, as the base pairing list was to
be repeated six times. So, in the overall event, some teams would still compete 
only $12$ times against each other, while other pairs compete $30$ times.
This difference can be reduced by using for instance three
times an optimized pairing list with $16$ flights (instead of six times
an optimized pairing list with $8$ flights).
We are able to solve the corresponding sailing league problem:

\begin{theorem}
For the parameters $\Nteams = 10$, $\Nflights = 16$ and $\Ninrace = 5$
an optimal tournament plan has a utility value $\lambda_{\max} - \lambda_{\min} = 2$.
A corresponding plan is contained in 
Table~\ref{tab:tournament-plan-10teams-16flights-5inrace}. 
\end{theorem}

For the computational proof of this result, we can first find a schedule 
with utility value $2$ as the one in
Table \ref{tab:tournament-plan-10teams-16flights-5inrace}.
A solver finds such a plan comparatively fast or other heuristic strategies can be used.
It is more difficult to prove that utility value $1$ is impossible.
All of the described linear and quadratic programs of Section~\ref{sec:programs}
did not finish on our computers in a reasonable amount of time.
However, a specific relaxation of the linear integer problem presented in Section \ref{sec:small_int_prog} allowed us to obtain a proof of optimality
for the tournament plan in Table \ref{tab:tournament-plan-10teams-16flights-5inrace}:
We allowed the binary variables $b_{ikl}$ 
of the $4$ teams $k=7,8,9,10$ (and all flights $i$)
to be real valued between $0$ and $1$.
Thus of the $720$ binary variables, 
we relaxed a total of $480$ variables to reals
(as we assume $l<k$ there are $30$ variables for each flight $i$).
As any relaxation gives a lower bound to the actual minimum of 
$b-a$, the computed exact minimum of $1.25$ in this relaxation
proves that $b-a\geq 2$ for the actual sailing league problem
(in which $b-a$ is an integer).

\bigskip

We finally note that the problem of finding an optimal
tournament plan for parameters $\Nteams = 10$, $\Ninrace = 5$
and $\Nflights = n\cdot 8$ with $n\geq 3$
($n$~times the base schedule of the Asian Pacific Champions League)
still appears to be open.

\subsection{German Sailing Leagues}

We analyzed a total of $20$ events of the German National Leagues, all with the usual parameters $\Nteams = 18$, $\Nflights = 16$ and $\Ninrace = 6$. 
The utility values for these events were either $2$ or $3$ indicating that the organizers put in some effort to obtain good tournament plans in the sense of the sailing league problem.
We ourselves conducted some heuristical searches for optimal tournament plans 
and tried a bunch of different strategies adapting or relaxing the quadratic and integer programs of Section~\ref{sec:programs}.
In all of these approaches we were not able to obtain a tournament plan 
$\lambda_{\max} - \lambda_{\min} = 1$.
As the average of pairwise competitions in the events of the German national 
Leagues is $\lambda=80/17=4.7058\ldots$ a tournament plan with 
$\lambda_{\max} - \lambda_{\min} = 1$ necessarily satisfies
$a=\lambda_{\max}=5$ and $b=\lambda_{\min}=4$.
We used this additional setting in the corresponding programs without success.

\bigskip

However, in contrast to the problems with a computational solution 
using one of the standard approaches of Section~\ref{sec:programs},
we found a solution using a suitable existence result from design theory:

\begin{theorem}
For the parameters $\Nteams = 18$, $\Nflights = 16$ and $\Ninrace = 6$
an optimal tournament plan has a utility value $\lambda_{\max} - \lambda_{\min} = 1$.
A corresponding plan is obtained by removing an arbitrary flight (row)
of Table~\ref{tab:tournament-plan-18teams-17flights-6inrace}
(cf. \cite{Kageyama_1983},\cite[Table 3]{baker1983resolvable_bibd_and_sols}).
\end{theorem}

A simple proof of this result is  
a direct application of 
Lemma \ref{lem:add_or_remove_rows} with $k=1$.

\subsection{European Sailing Champions League}

The largest problem we analyzed is the European Champions League in 2021
with $32$~slots for teams.
The actual number of teams competing in the event was~$29$, so that three team slots were marked \enquote{empty} in the corresponding tournament plan.
Analyzing the $72$ races of this event, we find that the $29$ competing teams had between $2$ and $7$ matches against each other. 
Thus in this case we may consider a kind of relaxed problem in which only a subset of 
the $\Nteams$ teams in the tournament plan is taken into account when computing the bounds
$a=\lambda_{\max}$ and $b=\lambda_{\min}$.

Unfortunately, we could not find an optimal pairing list for this problem.
Not even in the relaxed case using only $29$ teams or by other approaches.
To find at least a pairing list that yields a good upper bound, we used a perfect pairing list with parameters $\Nteams = 32$, $\Ninrace = 8$ and $\Nflights = 31$ (see also \cite[Lemma 5]{rokowska194resolvable_systems_of_8_tuples}) and then applied a heuristic search method to find the presumably best way to remove $13$ of the flights. The resulting tournament plan is given in Table \ref{tab:tournament-plan-32teams-18flights-8inrace}. Together with 
Corollary \ref{cor:no_perfect_pairing_list} this yields the following result:

\begin{theorem}
For the parameters $\Nteams = 32$, $\Nflights = 18$ and $\Ninrace = 8$
of the European Sailing Champions League Final 2021, an optimal pairing list satisfies
$$1\leq\lambda_{\max} - \lambda_{\min} \leq 3.$$
\end{theorem}

While we improved the utility value by two, compared with the actual event in 2021, 
so far we can not prove optimality of the tournament plan in 
 Table~\ref{tab:tournament-plan-32teams-18flights-8inrace}.
We finally note that other events of the European Sailing Champions League use various other parameters yielding even more unsolved sailing league problems with relatively big parameters.

\bigskip

We finally want to remark that a natural \emph{greedy approach} to 
solve a given sailing league problem does often not work.
It is tempting to think that one can get a quite good tournament plan
by optimizing the utility function for a few initial flights 
and then sequentially take a solution for a fixed number of flights and
optimize the tournament plan with $n$ additional flights.
This greedy strategy does in particular not work for the parameters of the 
European Sailing Champions League Final 2021.
Here we optimized for $4$ initial flights and obtained a tournament plan with 
utility value $2$. Then we added sequentially two flights to obtain
tournament plans with $6,8,\dots , 18$ flights.
This first lead to tournament plans with utility value $3$, but for $18$ flights the
optimal plan found in this way had utility value $4$ -- even if we 
counted only pairings of $29$ of the $32$ teams in our computations.
Hence, this gives an example where this natural greedy strategy 
produces a result worse than the one 
in Table~\ref{tab:tournament-plan-32teams-18flights-8inrace}.

\section{Additional connections and remarks}
\label{sec:relaxations}

In this section we discuss further information which, while interesting, have not been necessary for the discussion of the three use cases.

\subsection{More connections to design theory}

General pairing lists are strongly connected to the concept of \emph{partially balanced block designs}. 
In particular, a pairing list with parameters $\Nteams$, $\Nflights$, $\Ninrace$ and derived parameter $\lambda$ corresponds to a $(\Nteams, \Ninrace, \lambda)$-resolvable partially balanced block design $(T, R, \Lambda)$ with $m$ associate classes 
where 
$$m - 1 \leq \lambda_{\max} - \lambda_{\min} = \max(\Lambda) - \min(\Lambda).$$



In particular, pairing lists with $\lambda_{\max} - \lambda_{\min} = 1$ correspond to resolvable partially balanced block designs with $2$ associate classes.
More information about resolvable partially balanced block designs can be found, for example, in \cite{saurabh2021survey}.

To the best of our knowledge, resolvable partially balanced block designs are only known for very specific parameters and do not yield useful results for our three example cases.

\subsection{Relaxations}

During the work on this paper, we also thought about different ways of relaxing the definition of a sailing league problem, in order to obtain tournament plans in which the number of pairwise matches does not vary too much. 

The European Champions League in 2021 showed -- with three teams missing in the final -- 
that we could allow the races to be of different size for example,
as effectively done in the event. There are several results in the literature concerning block designs with multiple possible block sizes which could be useful in this setting.


Another possible relaxation is, to allow flights that are not a
partition of the teams. In particular, we could allow pairing lists
such that each flight lets every team compete, except a fixed number
of teams who skip this flight, i.e. the flight is just a partition of
a subset of the teams. If every flight has exactly one team that does
not participate, this is strongly connected to the concept of
\emph{nearly resolvable block designs} (see for example
\cite{Handbook_ChapterResolvableDesigns_2006} or \cite{banerjee2018-nearly-resolvable-designs}).
%
%
Some examples of pairing lists obtained in such a way are given in
Appendix~\ref{app:nearly_resolvable}
(Tables~\ref{tab:tournament-plan-13teams-4groups-nearly-resolvable} and~\ref{tab:tournament-plan-16teams-5groups-nearly-resolvable}).

Note that it is possible, to adapt the quadratic/linear programs of Section \ref{sec:programs} according to these relaxations.

\subsection{Related problem}

The sailing league problem discussed in this paper strictly requires the flights to be a partition of the teams, but principally allows different values for $\lambda(t, t')$. If we strictly require that the pairing list is perfect but allow that the flights might not partition the teams, we get the following related problem:


Let $\Nteams$, $\Ninrace$, $\lambda'$ be positive integer parameters.
A \emph{relaxed pairing list} (with respect to the given parameters) is a $(\Nteams, \Ninrace, \lambda)$-block design with $\lambda \geq \lambda'$ where additionally the races are organized in a list of flights which are partitions of respective subsets of the teams.
One could then ask to find an $(\Nteams, \Ninrace, \lambda')$-relaxed pairing list such that the number of flights is as small as possible.
%
%
In an event planned with this premise, each pair of teams would meet
\emph{exactly} the same number of times $\lambda$ in as few flights as possible.

\section{Conclusions}

In this paper we give an introduction to sailing league problems, depending on 
the number of competing teams, the number of flights (rounds) and the number of teams per race.
We find in particular that recently organized sailing events could be enhanced in terms of the discussed utility function. 
In very few cases of the parameters there exist perfect pairing lists, namely if and only if 
there exists a corresponding resolvable block design 
(see for instance Table \ref{tab:legal_params} and
Appendix~\ref{app:perfect_pairing_lists},
Tables~\ref{tab:tournament-plan-8teams-7flights-4inrace} to~\ref{tab:tournament-plan-20teams-19flights-5inrace}).
If, as in most cases, the parameters of a given sailing league do not 
allow the existence of a perfect pairing list, 
a combination of knowledge from design theory and implementations of linear and quadratic programs can be used to find good or even optimal pairing lists. 
In particular, the tournament plans presented in Appendix~\ref{app:best_known} and Table~\ref{tab:tournament-plan-18teams-17flights-6inrace} could be used for the parameters of the Asian Pacific Champions League, the German Sailing Leagues and the European Sailing Champions Leagues discussed in Section~\ref{sec:CaseStudies}.

Despite the findings of this paper, the sailing league problem in
general seems hard to solve and many parameters of real sailing events
have not been considered yet.
Furthermore, several relaxed and related problems discussed in
Sections~\ref{sec:DesignTheory} and~\ref{sec:relaxations} have not been studied sufficiently.
So further research is necessary to get a more complete picture.
We think in particular, further connections to combinatorial design theory and
more advanced techniques from operations research
could held to improve sailing league schedules in the future.

\section*{Acknowledgements}

We thank Gundram Leifert (ASV Warnem\"unde) for making us familiar
with sailing league problems.
We like to thank Michael Trick for very helpful remarks concerning
connections to operations research, and we like to thank two anonymous
referees for valuable links to the design theory literature.
Furthermore, we like to thank Frieder Ladisch, Valentin Dannenberg and Kerrin Bielser for helpful feedback on an earlier
version, and we like to thank a former class-mate of the first author for a very useful hint.

%
%


\vspace*{-0.5cm}


\renewcommand\refname{}  

\bigskip
\bigskip

\newpage

\appendix

\section{Perfect pairing lists}
\label{app:perfect_pairing_lists}

\begin{table}[h]

  \begin{tabular}{c|cccccccc}

    Flight/Team& 1& 2& 3& 4& 5& 6& 7& 8\\
    \hline
              1& 1& 1& 1& 1& 2& 2& 2& 2\\
              2& 2& 1& 1& 2& 1& 1& 2& 2\\
              3& 2& 2& 1& 1& 1& 2& 1& 2\\
              4& 1& 2& 2& 1& 1& 1& 2& 2\\
              5& 2& 1& 2& 1& 2& 1& 1& 2\\
              6& 1& 2& 1& 2& 2& 1& 1& 2\\
              7& 1& 1& 2& 2& 1& 2& 1& 2\\
  \end{tabular}

  \bigskip

  \caption{The tournament plan of a perfect pairing list for $\Nteams = 8$, $\Nflights = 7$, $\Ninrace = 4$.}
  \label{tab:tournament-plan-8teams-7flights-4inrace}

\end{table}

\begin{table}[h]

  \begin{tabular}{c|ccccccccc}

    Flight/Team& 1& 2& 3& 4& 5& 6& 7& 8& 9\\
    \hline
              1& 1& 1& 1& 2& 3& 2& 3& 2& 3\\
              2& 3& 2& 1& 1& 1& 2& 2& 3& 3\\
              3& 1& 3& 2& 3& 1& 1& 2& 2& 3\\
              4& 1& 2& 3& 1& 2& 3& 1& 2& 3\\
              5& 2& 2& 2& 1& 3& 1& 3& 1& 3\\
              6& 3& 1& 2& 2& 2& 1& 1& 3& 3\\
              7& 2& 3& 1& 3& 2& 2& 1& 1& 3\\
              8& 2& 1& 3& 2& 1& 3& 2& 1& 3\\
  \end{tabular}

  \bigskip

  \caption{The tournament plan of a perfect pairing list for $\Nteams = 9$, $\Nflights = 8$, $\Ninrace = 3$.}
  \label{tab:tournament-plan-9teams-8flights-3inrace}

\end{table}

\begin{table}[h]

  \begin{tabular}{c|cccccccccccc}

    Flight/Team& 1& 2& 3& 4& 5& 6& 7& 8& 9& 10& 11& 12\\
    \hline
              1& 1& 1& 1& 4& 2& 2& 2& 4& 3&  3&  3&  4\\
              2& 4& 1& 1& 1& 4& 2& 2& 2& 4&  3&  3&  3\\
              3& 1& 4& 1& 1& 2& 4& 2& 2& 3&  4&  3&  3\\
              4& 1& 1& 4& 1& 2& 2& 4& 2& 3&  3&  4&  3\\
              5& 2& 3& 4& 1& 2& 3& 4& 1& 2&  3&  4&  1\\
              6& 1& 4& 3& 2& 4& 1& 2& 3& 3&  2&  1&  4\\
              7& 4& 1& 2& 3& 3& 2& 1& 4& 1&  4&  3&  2\\
              8& 3& 2& 1& 4& 1& 4& 3& 2& 4&  1&  2&  3\\
              9& 1& 4& 3& 2& 4& 1& 2& 3& 3&  2&  1&  4\\
             10& 4& 1& 2& 3& 3& 2& 1& 4& 1&  4&  3&  2\\
             11& 3& 2& 1& 4& 1& 4& 3& 2& 4&  1&  2&  3\\
  \end{tabular}

  \bigskip

  \caption{The tournament plan of a perfect pairing list for $\Nteams = 12$, $\Nflights = 11$, $\Ninrace = 3$.}
  \label{tab:tournament-plan-12teams-11flights-3inrace}

\end{table}

\begin{table}[h]

  \begin{tabular}{c|cccccccccccc}

    Flight/Team& 1& 2& 3& 4& 5& 6& 7& 8& 9& 10& 11& 12\\
    \hline
              1& 1& 1& 1& 1& 2& 3& 2& 3& 3&  2&  3&  2\\
              2& 1& 2& 3& 1& 1& 1& 2& 3& 2&  3&  2&  3\\
              3& 1& 3& 2& 3& 2& 1& 1& 1& 2&  3&  3&  2\\
              4& 1& 2& 2& 2& 3& 3& 2& 1& 1&  1&  3&  3\\
              5& 1& 3& 1& 2& 2& 2& 3& 3& 2&  1&  1&  3\\
              6& 1& 3& 2& 1& 3& 2& 2& 2& 3&  3&  1&  1\\
              7& 1& 1& 3& 2& 3& 1& 3& 2& 2&  2&  3&  1\\
              8& 1& 1& 2& 3& 1& 2& 3& 1& 3&  2&  2&  3\\
              9& 1& 3& 3& 2& 1& 3& 1& 2& 3&  1&  2&  2\\
             10& 1& 2& 3& 3& 3& 2& 1& 3& 1&  2&  1&  2\\
             11& 1& 2& 1& 3& 2& 3& 3& 2& 1&  3&  2&  1\\
  \end{tabular}

  \bigskip

  \caption{The tournament plan of a perfect pairing list for $\Nteams = 12$, $\Nflights = 11$, $\Ninrace = 4$.}
  \label{tab:tournament-plan-12teams-11flights-4inrace}

\end{table}

\begin{table}[h]

  \begin{tabular}{c|cccccccccccc}

    Flight/Team& 1& 2& 3& 4& 5& 6& 7& 8& 9& 10& 11& 12\\
    \hline
              1& 1& 1& 1& 1& 1& 1& 2& 2& 2&  2&  2&  2\\
              2& 1& 2& 2& 1& 1& 2& 1& 1& 1&  2&  2&  2\\
              3& 1& 2& 1& 2& 1& 2& 2& 1& 2&  1&  1&  2\\
              4& 1& 1& 2& 1& 2& 2& 2& 1& 2&  1&  2&  1\\
              5& 1& 2& 2& 2& 1& 1& 1& 2& 2&  1&  2&  1\\
              6& 1& 2& 1& 2& 2& 1& 2& 1& 1&  2&  2&  1\\
              7& 1& 2& 2& 1& 2& 1& 2& 2& 1&  1&  1&  2\\
              8& 1& 1& 2& 2& 1& 2& 2& 2& 1&  2&  1&  1\\
              9& 1& 1& 2& 2& 2& 1& 1& 1& 2&  2&  1&  2\\
             10& 1& 1& 1& 2& 2& 2& 1& 2& 1&  1&  2&  2\\
             11& 1& 2& 1& 1& 2& 2& 1& 2& 2&  2&  1&  1\\
  \end{tabular}

  \bigskip

  \caption{The tournament plan of a perfect pairing list for $\Nteams = 12$, $\Nflights = 11$, $\Ninrace = 6$.}
  \label{tab:tournament-plan-12teams-11flights-6inrace}

\end{table}

\begin{table}[h]

  \begin{tabular}{c|cccccccccccccccc}

    Flight/Team& 1& 2& 3& 4& 5& 6& 7& 8& 9& 10& 11& 12& 13& 14& 15& 16\\
    \hline
              1& 1& 1& 1& 1& 2& 3& 4& 2& 3&  4&  2&  3&  4&  2&  3&  4\\
              2& 1& 2& 3& 4& 1& 1& 1& 2& 3&  4&  3&  4&  2&  4&  2&  3\\
              3& 1& 2& 3& 4& 2& 3& 4& 1& 1&  1&  4&  2&  3&  3&  4&  2\\
              4& 1& 2& 3& 4& 3& 4& 2& 4& 2&  3&  1&  1&  1&  2&  3&  4\\
              5& 1& 2& 3& 4& 4& 2& 3& 3& 4&  2&  2&  3&  4&  1&  1&  1\\
  \end{tabular}

  \bigskip

  \caption{The tournament plan of a perfect pairing list for $\Nteams = 16$, $\Nflights = 5$, $\Ninrace = 4$. Repeating this pairing list three times yields The tournament plan of a perfect pairing list for $\Nteams = 16$, $\Nflights = 15$, $\Ninrace = 4$.}
  \label{tab:tournament-plan-16teams-5flights-4inrace}

\end{table}

\begin{table}[h]

	\begin{tabular}{c|cccccccccccccccc}
		Flight/Team& 1& 2& 3& 4& 5& 6& 7& 8& 9& 10& 11& 12& 13& 14& 15& 16\\
		\hline
		1&1&2&1&2&1&2&1&2&1&2&1&2&1&2&1&2\\
		2&1&1&2&2&1&1&2&2&1&1&2&2&1&1&2&2\\
		3&1&2&2&1&1&2&2&1&1&2&2&1&1&2&2&1\\
		4&1&1&1&1&2&2&2&2&1&1&1&1&2&2&2&2\\
		5&1&2&1&2&2&1&2&1&1&2&1&2&2&1&2&1\\
		6&1&1&2&2&2&2&1&1&1&1&2&2&2&2&1&1\\
		7&1&2&2&1&2&1&1&2&1&2&2&1&2&1&1&2\\
		8&1&1&1&1&1&1&1&1&2&2&2&2&2&2&2&2\\
		9&1&2&1&2&1&2&1&2&2&1&2&1&2&1&2&1\\
		10&1&1&2&2&1&1&2&2&2&2&1&1&2&2&1&1\\
		11&1&2&2&1&1&2&2&1&2&1&1&2&2&1&1&2\\
		12&1&1&1&1&2&2&2&2&2&2&2&2&1&1&1&1\\
		13&1&2&1&2&2&1&2&1&2&1&2&1&1&2&1&2\\
		14&1&1&2&2&2&2&1&1&2&2&1&1&1&1&2&2\\
		15&1&2&2&1&2&1&1&2&2&1&1&2&1&2&2&1\\
	\end{tabular}
	
	\bigskip
	
	\caption{The tournament plan of a perfect pairing list for $\Nteams = 16$, $\Nflights = 15$, $\Ninrace = 8$.}
	\label{tab:tournament-plan-16teams-15flights-8inrace}
	
\end{table}

\begin{table}[h]

  \begin{tabular}{c|cccccccccccccccccc}

    Flight/Team& 1& 2& 3& 4& 5& 6& 7& 8& 9& 10& 11& 12& 13& 14& 15& 16& 17& 18\\
    \hline
              1& 1& 1& 1& 1& 1& 1& 2& 2& 2&  2&  2&  2&  3&  3&  3&  3&  3&  3\\
              2& 1& 2& 3& 1& 3& 2& 2& 3& 1&  2&  1&  3&  3&  1&  2&  3&  2&  1\\
              3& 1& 3& 2& 3& 1& 2& 2& 1& 3&  1&  2&  3&  3&  2&  1&  2&  3&  1\\
              4& 1& 3& 3& 2& 2& 1& 2& 1& 1&  3&  3&  2&  3&  2&  2&  1&  1&  3\\
              5& 1& 2& 1& 3& 2& 3& 2& 3& 2&  1&  3&  1&  3&  1&  3&  2&  1&  2\\
              6& 1& 1& 1& 1& 1& 1& 2& 2& 2&  2&  2&  2&  3&  3&  3&  3&  3&  3\\
              7& 1& 2& 3& 1& 3& 2& 2& 3& 1&  2&  1&  3&  3&  1&  2&  3&  2&  1\\
              8& 1& 3& 2& 3& 1& 2& 2& 1& 3&  1&  2&  3&  3&  2&  1&  2&  3&  1\\
              9& 1& 3& 3& 2& 2& 1& 2& 1& 1&  3&  3&  2&  3&  2&  2&  1&  1&  3\\
             10& 1& 2& 1& 3& 2& 3& 2& 3& 2&  1&  3&  1&  3&  1&  3&  2&  1&  2\\
             11& 1& 1& 2& 2& 3& 3& 1& 1& 2&  2&  3&  3&  1&  1&  2&  2&  3&  3\\
             12& 1& 3& 3& 2& 1& 2& 1& 3& 3&  2&  1&  2&  1&  3&  3&  2&  1&  2\\
             13& 1& 2& 1& 3& 3& 2& 1& 2& 1&  3&  3&  2&  1&  2&  1&  3&  3&  2\\
             14& 1& 2& 3& 1& 2& 3& 1& 2& 3&  1&  2&  3&  1&  2&  3&  1&  2&  3\\
             15& 1& 3& 2& 3& 2& 1& 1& 3& 2&  3&  2&  1&  1&  3&  2&  3&  2&  1\\
             16& 1& 1& 2& 2& 3& 3& 2& 2& 3&  3&  1&  1&  3&  3&  1&  1&  2&  2\\
             17& 1& 1& 2& 2& 3& 3& 2& 2& 3&  3&  1&  1&  3&  3&  1&  1&  2&  2\\
  \end{tabular}

  \bigskip

  \caption{The tournament plan of a perfect pairing list for $\Nteams = 18$, $\Nflights = 17$, $\Ninrace = 6$. Removing one flight gives us an optimal pairing list for the parameters of the German-Sailing-Leagues.}
  \label{tab:tournament-plan-18teams-17flights-6inrace}

\end{table}

\begin{table}[h]

  \begin{tabular}{c|cccccccccccccccccccc}

    Flight/Team& 1& 2& 3& 4& 5& 6& 7& 8& 9& 10& 11& 12& 13& 14& 15& 16& 17& 18& 19& 20\\
    \hline
              1& 1& 2& 5& 3& 1& 4& 5& 5& 4&  4&  1&  3&  3&  5&  2&  4&  3&  1&  2&  2\\
              2& 1& 3& 5& 3& 4& 1& 5& 5& 1&  4&  4&  2&  3&  5&  3&  4&  2&  1&  2&  2\\
              3& 1& 3& 5& 1& 4& 3& 5& 5& 3&  2&  4&  4&  1&  5&  3&  2&  4&  1&  2&  2\\
              4& 1& 5& 4& 4& 1& 3& 3& 1& 4&  5&  2&  5&  3&  2&  2&  4&  3&  5&  2&  1\\
              5& 1& 5& 1& 4& 4& 2& 3& 4& 1&  5&  3&  5&  3&  2&  3&  4&  2&  5&  2&  1\\
              6& 1& 5& 3& 2& 4& 4& 1& 4& 3&  5&  3&  5&  1&  2&  3&  2&  4&  5&  2&  1\\
              7& 1& 1& 3& 3& 5& 4& 4& 2& 5&  3&  1&  4&  5&  1&  2&  4&  3&  2&  5&  2\\
              8& 1& 4& 2& 3& 5& 1& 4& 3& 5&  3&  4&  1&  5&  1&  3&  4&  2&  2&  5&  2\\
              9& 1& 4& 4& 1& 5& 3& 2& 3& 5&  1&  4&  3&  5&  1&  3&  2&  4&  2&  5&  2\\
             10& 1& 1& 4& 5& 2& 5& 3& 1& 3&  3&  5&  4&  4&  2&  2&  4&  3&  2&  1&  5\\
             11& 1& 4& 1& 5& 3& 5& 3& 4& 2&  3&  5&  1&  4&  2&  3&  4&  2&  2&  1&  5\\
             12& 1& 4& 3& 5& 3& 5& 1& 4& 4&  1&  5&  3&  2&  2&  3&  2&  4&  2&  1&  5\\
             13& 5& 2& 4& 3& 2& 4& 3& 1& 3&  4&  1&  3&  4&  2&  5&  5&  5&  1&  1&  2\\
             14& 5& 3& 1& 3& 3& 1& 3& 4& 2&  4&  4&  2&  4&  2&  5&  5&  5&  1&  1&  2\\
             15& 5& 3& 3& 1& 3& 3& 1& 4& 4&  2&  4&  4&  2&  2&  5&  5&  5&  1&  1&  2\\
             16& 1& 1& 1& 1& 2& 4& 3& 5& 3&  4&  3&  2&  5&  2&  4&  2&  5&  3&  4&  5\\
             17& 1& 2& 5& 3& 3& 2& 4& 1& 1&  1&  5&  4&  2&  2&  4&  5&  3&  3&  5&  4\\
             18& 1& 3& 4& 5& 4& 2& 3& 2& 5&  4&  1&  1&  1&  5&  5&  2&  3&  2&  3&  4\\
             19& 5& 4& 2& 3& 5& 5& 5& 3& 4&  1&  1&  3&  2&  1&  2&  4&  1&  2&  3&  4\\
  \end{tabular}

  \bigskip

  \caption{The tournament plan of a perfect pairing list for $\Nteams = 20$, $\Nflights = 19$, $\Ninrace = 4$.}
  \label{tab:tournament-plan-20teams-19flights-4inrace}

\end{table}

\begin{table}[h]

  \begin{tabular}{c|cccccccccccccccccccc}

    Flight/Team& 1& 2& 3& 4& 5& 6& 7& 8& 9& 10& 11& 12& 13& 14& 15& 16& 17& 18& 19& 20\\
    \hline
              1& 1& 1& 2& 1& 1& 3& 2& 2& 3&  3&  4&  2&  4&  2&  3&  1&  4&  3&  4&  4\\
              2& 1& 4& 1& 2& 1& 1& 3& 2& 2&  3&  3&  4&  2&  4&  2&  3&  1&  4&  3&  4\\
              3& 1& 4& 4& 1& 2& 1& 1& 3& 2&  2&  3&  3&  4&  2&  4&  2&  3&  1&  4&  3\\
              4& 1& 3& 4& 4& 1& 2& 1& 1& 3&  2&  2&  3&  3&  4&  2&  4&  2&  3&  1&  4\\
              5& 1& 4& 3& 4& 4& 1& 2& 1& 1&  3&  2&  2&  3&  3&  4&  2&  4&  2&  3&  1\\
              6& 1& 1& 4& 3& 4& 4& 1& 2& 1&  1&  3&  2&  2&  3&  3&  4&  2&  4&  2&  3\\
              7& 1& 3& 1& 4& 3& 4& 4& 1& 2&  1&  1&  3&  2&  2&  3&  3&  4&  2&  4&  2\\
              8& 1& 2& 3& 1& 4& 3& 4& 4& 1&  2&  1&  1&  3&  2&  2&  3&  3&  4&  2&  4\\
              9& 1& 4& 2& 3& 1& 4& 3& 4& 4&  1&  2&  1&  1&  3&  2&  2&  3&  3&  4&  2\\
             10& 1& 2& 4& 2& 3& 1& 4& 3& 4&  4&  1&  2&  1&  1&  3&  2&  2&  3&  3&  4\\
             11& 1& 4& 2& 4& 2& 3& 1& 4& 3&  4&  4&  1&  2&  1&  1&  3&  2&  2&  3&  3\\
             12& 1& 3& 4& 2& 4& 2& 3& 1& 4&  3&  4&  4&  1&  2&  1&  1&  3&  2&  2&  3\\
             13& 1& 3& 3& 4& 2& 4& 2& 3& 1&  4&  3&  4&  4&  1&  2&  1&  1&  3&  2&  2\\
             14& 1& 2& 3& 3& 4& 2& 4& 2& 3&  1&  4&  3&  4&  4&  1&  2&  1&  1&  3&  2\\
             15& 1& 2& 2& 3& 3& 4& 2& 4& 2&  3&  1&  4&  3&  4&  4&  1&  2&  1&  1&  3\\
             16& 1& 3& 2& 2& 3& 3& 4& 2& 4&  2&  3&  1&  4&  3&  4&  4&  1&  2&  1&  1\\
             17& 1& 1& 3& 2& 2& 3& 3& 4& 2&  4&  2&  3&  1&  4&  3&  4&  4&  1&  2&  1\\
             18& 1& 1& 1& 3& 2& 2& 3& 3& 4&  2&  4&  2&  3&  1&  4&  3&  4&  4&  1&  2\\
             19& 1& 2& 1& 1& 3& 2& 2& 3& 3&  4&  2&  4&  2&  3&  1&  4&  3&  4&  4&  1\\
  \end{tabular}

  \bigskip

  \caption{A perfect tournament plan for $\Nteams = 20$, $\Nflights = 19$, $\Ninrace = 5$.}
  \label{tab:tournament-plan-20teams-19flights-5inrace}

\end{table}

\section{Pairing lists from nearly resolvable designs}
\label{app:nearly_resolvable}

\begin{table}[h]

	\begin{tabular}{c|ccccccccccccc}
		Description&1&2&3&4&5&6&7&8&9&10&11&12&13\\
		\hline
		1&1&1&2&3&4&2&4&3&2&1&3&4&0\\
		2&0&1&4&2&3&3&2&4&2&1&3&4&1\\
		3&1&0&3&4&2&4&3&2&2&1&3&4&1\\
		4&2&3&0&4&4&1&3&2&1&4&3&2&1\\
		5&2&3&4&0&4&3&2&1&2&4&1&3&1\\
		6&2&3&4&4&0&2&1&3&3&4&2&1&1\\
		7&2&3&1&3&2&0&4&4&1&4&2&3&1\\
		8&2&3&3&2&1&4&0&4&2&4&3&1&1\\
		9&2&3&2&1&3&4&4&0&3&4&1&2&1\\
		10&2&3&1&2&3&1&2&3&0&4&4&4&1\\
		11&1&1&2&2&2&3&3&3&4&0&4&4&1\\
		12&2&3&3&1&2&2&3&1&4&4&0&4&1\\
		13&2&3&2&3&1&3&1&2&4&4&4&0&1\\
	\end{tabular}
	
	\bigskip

	\caption{A pairing list obtained from a nearly resolvable design for $\Nteams = 13$, $\Nflights = 13$, $\Ninrace = 3$.
	A zero indicates that the respective team skips the respective flight.}
	\label{tab:tournament-plan-13teams-4groups-nearly-resolvable}

\end{table}

\begin{table}[h]

	\begin{tabular}{c|cccccccccccccccc}
		Flight/Team&1&2&3&4&5&6&7&8&9&10&11&12&13&14&15&16\\
		\hline
		1&1&1&1&2&3&4&5&2&5&3&4&2&4&5&3&0\\
		2&0&1&1&2&3&4&5&4&3&5&2&5&3&2&4&1\\
		3&1&0&1&2&3&4&5&5&2&4&3&3&5&4&2&1\\
		4&1&1&0&2&3&4&5&3&4&2&5&4&2&3&5&1\\
		5&2&3&4&0&5&5&5&1&3&4&2&1&4&2&3&1\\
		6&2&3&4&5&0&5&5&4&2&1&3&3&2&4&1&1\\
		7&2&3&4&5&5&0&5&2&4&3&1&4&1&3&2&1\\
		8&2&3&4&5&5&5&0&3&1&2&4&2&3&1&4&1\\
		9&2&3&4&1&4&2&3&0&5&5&5&1&3&4&2&1\\
		10&2&3&4&3&2&4&1&5&0&5&5&4&2&1&3&1\\
		11&2&3&4&4&1&3&2&5&5&0&5&2&4&3&1&1\\
		12&2&3&4&2&3&1&4&5&5&5&0&3&1&2&4&1\\
		13&2&3&4&1&3&4&2&1&4&2&3&0&5&5&5&1\\
		14&2&3&4&4&2&1&3&3&2&4&1&5&0&5&5&1\\
		15&2&3&4&2&4&3&1&4&1&3&2&5&5&0&5&1\\
		16&2&3&4&3&1&2&4&2&3&1&4&5&5&5&0&1\\
	\end{tabular}
	
	\bigskip

	\caption{A pairing list obtained from a nearly resolvable design for $\Nteams = 16$, $\Nflights = 16$ and $\Ninrace = 3$. A zero indicates that the respective team skips the respective flight.}
	\label{tab:tournament-plan-16teams-5groups-nearly-resolvable}

\end{table}

%
%
%


\newpage

\section{Tournament plans with best known utility}
\label{app:best_known}

\begin{table}[h]
  \setlength{\tabcolsep}{1pt}
  \begin{tabular}{c|cccccccccccccccccccccccccccccccc}

    Flight/Team& 1& 2& 3& 4& 5& 6& 7& 8& 9& 10& 11& 12& 13& 14& 15& 16& 17& 18& 19& 20& 21& 22& 23& 24& 25& 26& 27& 28& 29& 30& 31& 32\\
    \hline
              1& 1& 1& 1& 1& 1& 1& 1& 1& 2&  2&  2&  2&  2&  2&  2&  2&  3&  3&  3&  3&  3&  3&  3&  3&  4&  4&  4&  4&  4&  4&  4&  4\\
              2& 1& 1& 1& 1& 1& 1& 1& 1& 2&  2&  2&  2&  2&  2&  2&  2&  3&  3&  3&  3&  3&  3&  3&  3&  4&  4&  4&  4&  4&  4&  4&  4\\
              3& 1& 1& 2& 2& 3& 3& 4& 4& 1&  1&  3&  3&  4&  4&  2&  2&  1&  1&  4&  4&  2&  2&  3&  3&  1&  1&  2&  2&  3&  3&  4&  4\\
              4& 1& 1& 2& 2& 3& 3& 4& 4& 3&  3&  1&  1&  2&  2&  4&  4&  4&  4&  1&  1&  3&  3&  2&  2&  2&  2&  1&  1&  4&  4&  3&  3\\
              5& 1& 1& 2& 2& 3& 3& 4& 4& 4&  4&  2&  2&  1&  1&  3&  3&  2&  2&  3&  3&  1&  1&  4&  4&  3&  3&  4&  4&  1&  1&  2&  2\\
              6& 1& 1& 2& 2& 3& 3& 4& 4& 2&  2&  4&  4&  3&  3&  1&  1&  3&  3&  2&  2&  4&  4&  1&  1&  4&  4&  3&  3&  2&  2&  1&  1\\
              7& 1& 2& 1& 2& 3& 4& 3& 4& 1&  3&  1&  3&  4&  2&  4&  2&  1&  4&  1&  4&  2&  3&  2&  3&  1&  2&  1&  2&  3&  4&  3&  4\\
              8& 1& 2& 1& 2& 3& 4& 3& 4& 3&  1&  3&  1&  2&  4&  2&  4&  4&  1&  4&  1&  3&  2&  3&  2&  2&  1&  2&  1&  4&  3&  4&  3\\
              9& 1& 2& 1& 2& 3& 4& 3& 4& 4&  2&  4&  2&  1&  3&  1&  3&  2&  3&  2&  3&  1&  4&  1&  4&  3&  4&  3&  4&  1&  2&  1&  2\\
             10& 1& 2& 1& 2& 3& 4& 3& 4& 2&  4&  2&  4&  3&  1&  3&  1&  3&  2&  3&  2&  4&  1&  4&  1&  4&  3&  4&  3&  2&  1&  2&  1\\
             11& 1& 2& 2& 1& 3& 4& 4& 3& 1&  3&  3&  1&  4&  2&  2&  4&  1&  4&  4&  1&  2&  3&  3&  2&  1&  2&  2&  1&  3&  4&  4&  3\\
             12& 1& 2& 2& 1& 3& 4& 4& 3& 3&  1&  1&  3&  2&  4&  4&  2&  4&  1&  1&  4&  3&  2&  2&  3&  2&  1&  1&  2&  4&  3&  3&  4\\
             13& 1& 2& 2& 1& 3& 4& 4& 3& 4&  2&  2&  4&  1&  3&  3&  1&  2&  3&  3&  2&  1&  4&  4&  1&  3&  4&  4&  3&  1&  2&  2&  1\\
             14& 1& 2& 2& 1& 3& 4& 4& 3& 2&  4&  4&  2&  3&  1&  1&  3&  3&  2&  2&  3&  4&  1&  1&  4&  4&  3&  3&  4&  2&  1&  1&  2\\
             15& 1& 2& 3& 4& 1& 2& 3& 4& 2&  4&  3&  1&  2&  4&  3&  1&  3&  2&  4&  1&  3&  2&  4&  1&  4&  3&  2&  1&  4&  3&  2&  1\\
             16& 1& 2& 3& 4& 2& 1& 4& 3& 2&  4&  3&  1&  4&  2&  1&  3&  3&  2&  4&  1&  2&  3&  1&  4&  4&  3&  2&  1&  3&  4&  1&  2\\
             17& 1& 2& 3& 4& 3& 4& 1& 2& 4&  2&  1&  3&  1&  3&  4&  2&  2&  3&  1&  4&  1&  4&  2&  3&  3&  4&  1&  2&  1&  2&  3&  4\\
             18& 1& 2& 3& 4& 4& 3& 2& 1& 2&  4&  3&  1&  1&  3&  4&  2&  3&  2&  4&  1&  1&  4&  2&  3&  4&  3&  2&  1&  1&  2&  3&  4\\
  \end{tabular}

  \bigskip

  \caption{A tournament plan for $\Nteams = 32$, $\Nflights = 18$, $\Ninrace = 8$ with $\lambda_{\max} - \lambda_{\min} = 3$.}
  \label{tab:tournament-plan-32teams-18flights-8inrace}

\end{table}

\begin{table}[h]

  \begin{tabular}{c|cccccccccc}

    Flight/Team& 1& 2& 3& 4& 5& 6& 7& 8& 9& 10\\
    \hline
              1& 1& 2& 1& 2& 1& 1& 2& 1& 2&  2\\
              2& 1& 1& 2& 2& 1& 2& 1& 1& 2&  2\\
              3& 1& 1& 1& 1& 2& 2& 2& 2& 1&  2\\
              4& 1& 1& 2& 1& 1& 2& 2& 2& 2&  1\\
              5& 1& 2& 1& 1& 2& 2& 2& 1& 2&  1\\
              6& 1& 2& 2& 1& 2& 1& 1& 1& 2&  2\\
              7& 1& 1& 2& 2& 2& 1& 2& 1& 1&  2\\
              8& 1& 1& 2& 2& 2& 1& 1& 2& 2&  1\\
              9& 1& 2& 2& 1& 1& 1& 2& 2& 1&  2\\
             10& 1& 2& 1& 2& 1& 2& 1& 2& 1&  2\\
             11& 1& 2& 2& 2& 1& 2& 2& 1& 1&  1\\
             12& 1& 1& 2& 1& 2& 1& 2& 2& 1&  2\\
             13& 1& 2& 1& 2& 2& 1& 2& 2& 1&  1\\
             14& 1& 2& 2& 1& 2& 2& 1& 1& 1&  2\\
             15& 1& 1& 1& 1& 2& 2& 1& 2& 2&  2\\
             16& 1& 1& 2& 2& 2& 1& 2& 1& 2&  1\\
  \end{tabular}

  \bigskip

  \caption{A tournament plan for $\Nteams = 10$, $\Nflights = 16$,
    $\Ninrace = 5$ with $\lambda_{\max} - \lambda_{\min} = 2$.
    The first eight rows are an optimal tournament plan for $\Nteams = 10$, $\Nflights = 8$, $\Ninrace = 5$.}
  \label{tab:tournament-plan-10teams-16flights-5inrace}

\end{table}

\end{document}